\documentclass[12pt,bezier]{article}
\usepackage{amssymb}
\usepackage{mathrsfs}
\usepackage{amsmath}
\usepackage{amsfonts,amsthm,amssymb}
\usepackage{amsfonts}
\usepackage{graphics}
\usepackage{graphicx}
\usepackage{color}
\usepackage{ifpdf}
\usepackage{fancyhdr}
\usepackage{epstopdf}
\usepackage{epsfig}
\usepackage{indentfirst}
\usepackage{CJK}

\newtheorem{lem}{Lemma}[section]
\newtheorem{thm}[lem]{Theorem}

\begin{document}
	\title{An excluded minor theorem for the Wagner graph plus an edge}
	\author{ Yuqi Xu, \quad Weihua Yang\footnote{Corresponding author. E-mail: ywh222@163.com,~yangweihua@tyut.edu.cn}\\
		\\ \small Department of Mathematics, Taiyuan University of
		Technology,\\
		\small  Taiyuan Shanxi-030024,
		China\\}
	\date{}
	\maketitle

{\small{\bf Abstract:}
Let $ V_{8}+e $ denote the unique graph obtained from the Wagner graph, also known as $ V_{8} $, by adding an edge between two vertices of distance $ 3 $ on the Hamilton cycle, which is exactly a split of a minor of the Petersen graph. A complete characterization of all internally 4-connected graphs with no $ V_{8} $ minor is given in [J. Maharry and N. Robertson, The structure of graphs not topologically containing the Wagner graph, J. Combin. Theory Ser. B 121 (2016) 398-420]. In this paper we characterize all internally $ 4 $-connected graphs with no $ V_{8}+e $ minor.


\section{Introduction}

All graphs in this article are simple. Let $ G $ and $ H $ be two graphs. $ H $ is called a $minor$ of $ G $ if $ H $ is obtained from $ G $ by deleting edges or contracting edges, denoted by $ H\preceq G $. $ G $ is $H$-$minor$-$free$ if no minor of $ G $ is isomorphic to $ H $. For a given graph $ H $, $ H $-minor-free graphs play an important role in graph theory. Determining $ K_{6} $-minor-free graphs and Petersen-minor-free graphs are the two most famous problems in this area.

Let $ k $ be a non-negative integer. A $ k $-\textit{separation} of a graph $ G $ is an unordered pair $\left\{G_{1}, G_{2}\right\}$ of induced subgraphs of $ G $ such that $ V(G_{1})\cup V(G_{2})=V(G), $ $ E(G_{1})\cup E(G_{2})=E(G) $, $ V (G_{1})-V (G_{2})\neq\emptyset $, $ V(G_{2})-V(G_{1})\neq\emptyset $, and $ |V(G_{1})\cap V(G_{2})|=k $. If $ G $ has a $ k $-separation, then there is $ X\subseteq V(G) $ such that $ |X|=k $ and $ G\setminus X $ has at least two components. A 3-connected graph $ G $ on at least five vertices is said to be \textit{quasi} $ 4 $-\textit{connected} if for every 3-separation $\left\{G_{1}, G_{2}\right\}$ of $ G $, one of $ G_{1} $ or $ G_{2} $ contains exactly 4 vertices. We define $ G $ to be \textit{internally} $ 4 $-\textit{connected} if for every 3-separation $\left\{G_{1}, G_{2}\right\}$ where $ G_{1}\cap G_{2}=\left\{x, y, z\right\} $, and there are no edges among $ x, y $ and $ z $.

Let $ G\setminus e $ denote the graph obtained from $ G $ by deleting an edge $ e $. The reverse operation of deleting an edge is adding an edge, that is $ G $ obtained from $ G\setminus e $ by adding edge $ e $. We use $ G/e $ denote the graph obtained from $ G $ by first contracting an edge $ e $ then deleting all but one edge from each parallel family. The reverse operation of contracting an edge is splitting a vertex. To be precise, suppose $ v $ is a vertex with degree at least four in a graph $ G $. Let $ N_{G}(v) $ denote the set of neighbors of $ v $, which are vertices adjacent to $ v $. Let $ X, Y\subseteq N_{G}(v) $ such that $ X\cup Y = N_{G}(v) $ and $ |X|, |Y|\geq2 $. The splitting $ v $ results in the new graph $ G^{'} $ obtained from $ G\setminus v $ by adding two new adjacent vertices $ x, y $ then joining $ x $ to all vertices in $ X $ and $ y $ to all vertices in $ Y $. We call $ G^{'} $ a \textit{split} of $ G $. Tutte \cite{W.T. Tutte} stated that every 3-connected graph can be obtained from a wheel by repeatedly adding edges and splitting vertices, this result is known as Tutte's Wheel Theorem. 

For each 3-connected graph $ H $ with at most 11 edges, $ H $-free graphs are characterized \cite{G. Ding}. There are three classical results for 3-connected graph with 12 edges: Maharry \cite{J. Maharry Cube} characterized $Cube$-minor-free graphs, Ding \cite{G. Ding Oct} characterized $Oct$-minor-free graphs, and Robertson \cite{J. Maharry and N. Robertson} characterized $V_{8}$-minor-free graphs.

To state the theorem we need to define a few classes of graphs. For any graph $ G $, the \textit{line graph} of $ G $, denoted by $ L(G) $, is a graph such that each vertex of $ L(G) $ represents an edge of $ G $, and two vertices of $ L(G) $ are adjacent if and only if their corresponding edges share a common end vertex in $ G $. For each integer $ n\geq3 $, a \textit{double}-\textit{wheel}, $ DW_{n} (n\geq3) $, is a graph on $ n+2 $ vertices obtained from a cycle $ C_{n} $ by adding two nonadjacent vertices $ u, v $ and joining them to all vertices on the cycle. An \textit{alternating double}-\textit{wheel} $ AW_{2n} $ is a subgraph of $ DW_{2n} (n\geq3) $ such that $ u $ and $ v $ are alternately adjacent to every vertex in $ C_{2n} $. Notice that $ AW_{6} $ is a cube, see Figure 1. For each integer $ n\geq3 $, let $ DW_{n}^{+} $ and $ AW_{2n}^{+} $ be graphs obtained from $ DW_{n} $ and $ AW_{2n} $, respectively, by joining $ u $ and $ v $. Let $ \mathcal{D}^{+}=\left\{DW_{n}^{+}: n\geq3\right\}\cup\left\{AW_{2n}^{+}: n\geq3\right\}. $ Then every graph in $ \mathcal{D}^{+} $ is nonplanar.

The $ n $-\textit{rung ladder}, $ L_{n} $, has vertices $ v_{1}, v_{2}, ..., v_{n} $ and $ u_{1}, u_{2}, ..., u_{n} $, where $ v_{1}, v_{2}, ..., v_{n} $ and $ u_{1}, u_{2}, ..., u_{n} $ form paths and each $ u_{i} $ is adjacent to $ v_{i} $ for $ i = 1, 2, ..., n $. The $ n $-\textit{rung Mobius ladder}, $ M_{n} $, is obtained from $ L_{n} $ by adding edges joining $ v_{1} $ to $ u_{n} $ and $ v_{n} $ to $ u_{1} $. The graph $ SM_{n} $ is obtained from $ DW_{n} $ by adding some number of triads to triangles of $ DW_{n} $ such that every triad is adjacent to at least two $ r_{i} $ and no two triads are adjacent to the same pair of $ r_{i} $. Let $ \mathbb{K}_{4,n} $ be the collection of all $ q $-$ 4 $-$ c $ minors of any $ K_{4,n} $. Let $ \mathbb{M}_{n} $ be the collection of all $ q $-$ 4 $-$ c $ minors of $ M_{n} $, and let $ \mathbb{SM}_{n} $ be the collection of all $ q $-$ 4 $-$ c $ minors of $ SM_{n} $.

\begin{thm}[\cite{J. Maharry and N. Robertson}]\label{thm1.1}
Every internally 4-connected $ V_{8} $-minor-free graph $ G $ satisfices one of the following conditions:
	
(i) $ G $ is planar,

(ii) $ |G|\leq7 $,

(iii) $ G\cong L(K_{3,3}) $,

(iv) $ G\setminus\left \{w,x,y,z\right \} $ has no edges for some $ w,x,y,z\in V(G) $,

(v) $ G\in \mathcal{D}^{+} $.
\end{thm}

Our main tool is a chain theorem of Chun, Mayhew and Oxley \cite{C. Chun}, which will be useful in creating a process that generates all internally 4-connected graphs. To explain this result we need a few definitions. For each integer $ n\geq 5 $, let $ C_{n}^{2} $ be a graph obtained from a cycle $ C_{n} $ by joining all pairs of vertices of distance two on the cycle. Notice that $ C_{5}^{2}=DW_{3}^{+}=K_{5} $, see Figure \ref{t1}. Let \textit{terrahawk} be the graph shown in Figure 1, which can be obtaind from a cube by adding a new vertex and joining it to four vertices in the same $ C_{4} $.

\input{1.tpx}

\begin{thm}[\cite{C. Chun}]\label{thm1.2}
Let $ G $ be an internally 4-connected graph such that $ G $ is not $ K_{3,3} $, terrahawk, $ C_{n}^{2} $ $ (n\geq5) $, or $ AW_{2n} $ $ (n\geq3) $. Then $ G $ has an internally 4-connected minor $ H $ with $ 1\leq ||G||-||H||\leq3 $.
\end{thm}

Equivalently, this theorem says that, for every internally 4-connected graph $ G $, there exists a sequence of internally 4-connected graphs $ H_{0}, H_{1}, H_{2}, ... , H_{k} $ such that

(i) $ H_{k} \cong G $ and $ H_{0} $ is $ K_{3,3} $, terrahawk, $ C_{n}^{2} $ $ (n\geq5) $, or $ AW_{2n} $ $ (n\geq3) $, and

(ii) $ H_{i} (i = 2, ... , k) $ is obtained from $ H_{i-1} $ by adding edges or splitting vertices at most three times.

Let $ G $ be an internally 4-connected graph. It is not difficult to show that: if $ G+e $ is not internally 4-connected, then the two ends of $ e $ are neighbors of a cubic vertex; if splitting a vertex of $ G $ results in a graph $ G^{'} $ that is not internally 4-connected, then one of the new vertices is cubic and is contained in a triangle. For above two cases, we mainly generate all internally 4-connected graphs by adding an edge, reducing the cubic vertices to new ones of degree 4, or splitting the neighbors of the cubic vertices.

The class of graphs with no minor isomorphic to Petersen, we will denote by $ P_{0} $, has been widely studied. But there is no known exact structual characterization. However, the problem becomes a bit more manageable from the other direction, specifically, when we consider the graphs that are minors of $ P_{0} $. Naturally, if a graph $ G $ were a minor of $ P_{0} $, and a third graph $ H $ did not contain $ G $ as a minor, we could safely say that $ H $ does not contain $ P_{0} $ either. Using the Splitter Theorem \cite{P. Seymour}, we may grow $ K_{5} $ to $ P_{0} $ creating the sequence of graphs in Figure \ref{t2}.

\input{2.tpx}

For 3-connected graph with 13 edges, only $Oct+e$ and $V_{8}+f$, isomorphic to $ P_{2} $, are characterized. 

\begin{thm}[\cite{J. Maharry}]
Every 4-connected graph that does not contain a minor isomorphic to $ Oct^{+} $ is either planar or the square of an odd cycle.
\end{thm}

\begin{thm}[\cite{A.B. Ferguson}]
For every integer $ n\geq6 $, there exists a number $ N $ such that every non-planar $ q $-$ 4 $-$ c $ graph $ G $ of order at least $ N $ contains a $ P_{2} $ minor, unless $ G $ is a member of $ \mathbb{K}_{4,n} $, $ \mathbb{M}_{n} $, or $ \mathbb{SM}_{n} $.
\end{thm}

In this paper, we characterize another graph obtained by joining two vertices of distance $ 3 $ on the cycle from $ V_{8} $, denoted by $ V_{8}+e $ and shown in Figure 2. Note that it has 13 edges and it is exactly a split of $ P_{3} $.

Let $ \mathcal{E}(AW_{6}^{+}) $ denote the set of graphs obtained from $ AW_{6}^{+} $ by adding edges not construct an $ X $ (show in Figure \ref{t3}). The following is the main theorem of this article.

\begin{thm}\label{thm1.3}
Every internally 4-connected $ V_{8}+e $-minor-free graph $ G $ satisfices one of the following conditions:

(i) $ G $ is planar, 

(ii) $ |G|\leq7 $,

(iii) $ G\in \mathcal{E}(V_{8}+f) \cup \mathcal{E}(AW_{6}^{+}) $,

(iv) $ G $ is a minor of $ \Gamma $ and $ \Gamma_{2} $ in the Figure \ref{t24} or $ G\in \left\{AW_{2n}^{+}: n\geq4\right\} $.

\end{thm}

\input{24.tpx}

We mainly construct the graphs based on $ AW_{2n} $ $ (n\geq3) $, terrahawk, $ K_{3,3} $, and $ C_{n}^{2} $ $ (n\geq5) $  in Sections 2, 3 and 4 respectively, from which the Theorem 1.5 follows.

\section{Extension of $ AW_{2n} $}

In this section, we shall characterize the internally 4-connedted $ V_{8}+e $-minor-free graphs which are obtained from $ AW_{2n} $ by repeatedly adding edges or splitting vertices. Note that we always suppose that the new vertices reduced in a split have degree 3, since the other splits contain such a split as a minor, so dose latter in this article. 

\begin{lem}\label{lem2}
Let $ G\succeq AW_{2n} $ be an internally 4-connected graph. Then $ G $ is $ V_{8}+e $-minor-free if and only if $ G\in \mathcal{E}(AW_{6}^{+})\cup\left\{AW_{2n}^{+}: n\geq4\right\} $. 
\end{lem}

\begin{proof}
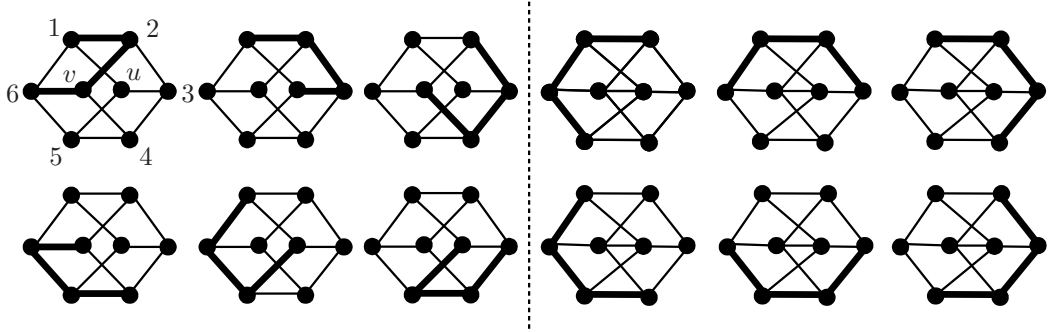
Observe that $ AW_{2n} $ is planar, thus is $ V_{8}+e $-minor-free. Since every graph in $ AW_{2n} $ contains $ AW_{6} $ as a minor, suppose that $ H_{0}\cong AW_{6} $. And $ AW_{6} $ contains no vertex of degree 4, we first suppose that $ H_{1} $ is obtained from $ AW_{6} $ by adding edges. The unique non-planar graph generated is $ AW_{6}^{+} $. To continue constructing other graphs, we first state the following result.

\noindent\textbf{Claim 1. }\textit{Let $ P $ be a 3-path with vertex set $\left\{ a_{1}, a_{2}, a_{3}, a_{4}\right\}$ in a ladder (see Figure \ref{t3}). Then adding an $ X $ to corners $ a_{1}, a_{2}, a_{3}, a_{4} $ will reduce to a $ V_{8}+e $-minor.
}
\input{3.tpx}

\input{4.tpx}

It can be seen from Figure \ref{t4} that adding an $ X $ to any such path, say path $ P_{12v6} $ in $ AW_{6} $ results in a $ V_{8}+e $-minor. Let $ \mathcal{F}(G)=\left\{e:G+e\succeq V_{8}+e\right\} $, which is the set of \textit{forbidden edges}. Then by symmetry, we have $ \big\{\left\{15,6u\right\}, \left\{1v,26\right\}, \left\{13,2u\right\}, \left\{24,3v\right\}, \left\{35,4u\right\}, \left\{46,5v\right\}\big\} \subseteq \mathcal{F}(AW_{6}) $. Therefore, any non-planar graph, with an exception $ AW_{6}^{+} $, is not $ V_{8}+e $-minor-free.

Thus suppose that $ H_{1}\cong AW_{6}^{+} $. Then by Theorem 1.1, $ AW_{6}^{+} $ is $ V_{8}+e $-minor-free. And clearly $ \mathcal{F}(AW_{6})\subseteq \mathcal{F}(AW_{6}^{+}) $. Similar to the analysis above, we also state that the edge pairs $ \big\{\left\{15,26\right\}, \left\{15,46\right\}, \left\{13,26\right\}, \left\{13,24\right\}, \left\{24,35\right\}, \left\{35,46\right\}\big\}\subseteq \mathcal{F}(AW_{6}^{+}) $ (see Figure \ref{t4}). 

Now we show that only such an addition results a $ V_{8}+e $-minor. It equals that any other addition resulting a $ V_{8}+e $-minor contains such an addition $ X $. Let $ M $ be a $ V_{8}+e $-minor-free graph obtained from $ AW_{6}^{+} $ by repeatly adding edges with $ \left|E(M)\right| $ maximal. The followings are the properties of  $ M $.

\textbf{P1}. $ M $ must contain some vertex of degree 7.

Let $ N_{1} $ be a maximal $ V_{8}+e $-minor-free graph, and each vertex of $ N_{1} $ has degree as large as possible. However, $ N_{1} $ with an additional edge $ 2u $ is still $ V_{8}+e $-minor-free, contradicting the maximality of $ N_{1} $. 

\textbf{P2}. There are no two  adjacent vertices of degree 7 in $ M $, except for the two vertices $ u $ and $ v $.

It is obviously that if two successive vertices are both of degree 7, say 4 and 5 (see $ N_{2} $ in Figure 6), then vertices 4, 5 and their neighbors will reduce a required path $ P_{3456} $. So dose $ N_{3} $ with $ d(v)=d(4)=7 $. This result deduces that $ M $ could not contain more than 5 vertices of degree 7. 

From above two properties, we can suppose that $ M $ contains at most 4 vertices of degree 7. Now we characterize all graphs $ M $.

If $ M $ contains 4 vertices of degree 7, then by symmetry, we have $ d(v)=d(1)=d(3)=d(5)=7 $, generating the graph $ M_{1} $.

If $ M $ contains 3 vertices of degree 7, then by symmetry, we have $ d(v)=d(1)=d(3)=7 $ or $ d(1)=d(3)=d(5)=7 $, both of which require another vertex of degree 7, isomorphic to $ M_{1} $.

If $ M $ contains 2 vertices of degree 7, then we obtain the graph $ M_{2} $ with $ d(u)=d(v)=7 $, the graph $ M_{3} $ with $ d(v)=d(1)=7 $ and generate a graph isomorphic to $ M_{1} $ with $ d(1)=d(3)=7 $. 

If $ M $ contains one vertices of degree 7, then we obtain the graph $ M_{4} $ with $ d(v)=7 $, and generate a graph isomorphic to $ M_{1} $ with $ d(1)=7 $.

\input{21.tpx}

Observe that any graph $ M_{i} $ with an additional edge contains such an addition $ X $ as a minor (see Figure \ref{t21}). Thus the statement of claim is true. 

Let $ \mathcal{E}(AW_{6}^{+}) $ denote the set of graphs obtained by adding edges (not in $ \mathcal{F}(AW_{6}^{+}) $) to $ AW_{6}^{+} $. Then any graph $ G\in \mathcal{E}(AW_{6}^{+}) $ is $ V_{8}+e $-minor-free. 

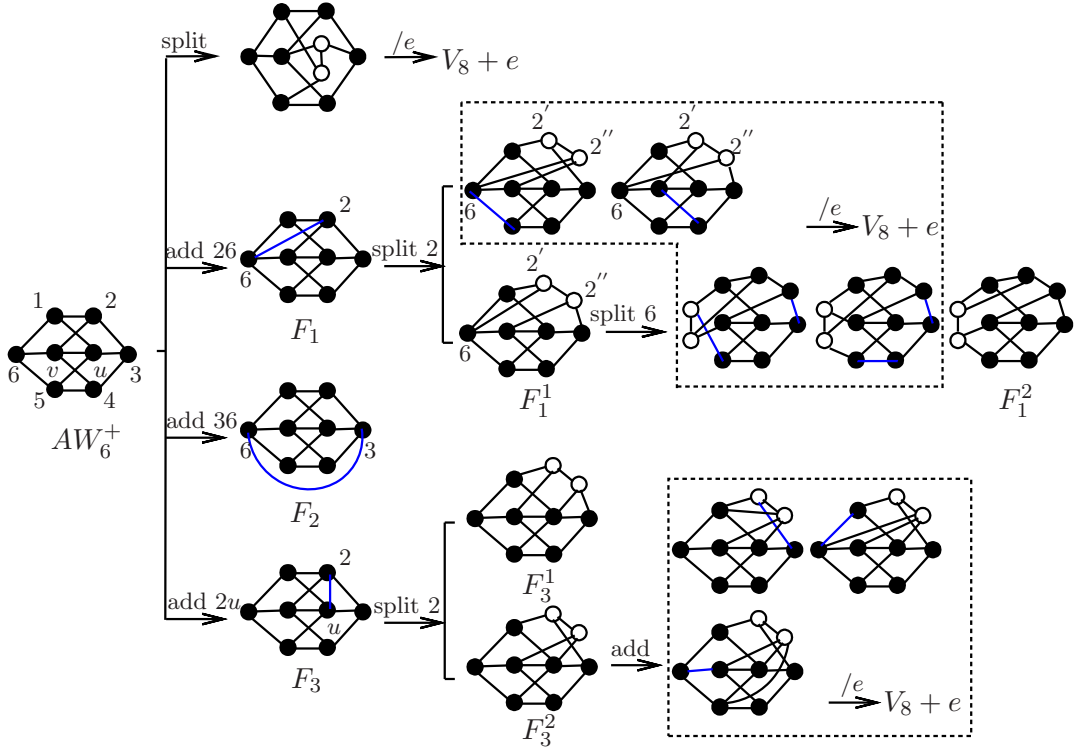
Next we consider the split of graphs in $ \mathcal{E}(AW_{6}^{+}) $. Firstly the unique split of $ AW_{6}^{+} $ (up to symmetry) contains a $ V_{8}+e $-minor. Secondly by symmetry we obtain three not internally 4-connected ($ i $-$ 4 $-$ c $) graphs $ F_{1} $ (obtained by adding 26), $ F_{2} $ (obtained by adding 36) and $ F_{3} $ (obtained by adding $ 2u $) from $ AW_{6}^{+} $. For $ F_{1} $, splitting vertex 2 results in two graphs with $ V_{8}+e $ minors and a not $ i $-$ 4 $-$ c $ graph $ F_{1}^{1} $. Since $ F_{1}^{1} $ contains two cubic vertices in a triangle, no addition is possible, and splitting vertex 6 generates two graphs with $ V_{8}+e $ minors and a not $ i $-$ 4 $-$ c $ graph $ F_{1}^{2} $, with 3 cubic vertices in a triangle, thus the process terminates, and we denote the process based on $ F_{1} $ by $ \mathcal{P}_{1} $. For $ F_{2} $, any split of it contains the split of $ AW_{6}^{+} $ as a minor, thus contains a $ V_{8}+e $-minor. For $ F_{3} $, splitting vertex 2 generates two not $ i $-$ 4 $-$ c $ graphs $ F_{3}^{1} $ and $ F_{3}^{2} $, then any split of $ F_{3}^{1} $ and $ F_{3}^{2} $ contains the split of $ AW_{6}^{+} $ as a minor. On the other hand, no addition to $ F_{3}^{1} $ is possble, since $ F_{3}^{1} $ has two cubic vertices in a triangle. And we can deduce by symmetry that any addition to $ F_{3}^{2} $ gives rise to a $ V_{8}+e $-minor, we denote the process based on $ F_{3} $ by $ \mathcal{P}_{2} $. Note that the additions of $ F_{i} $ are in $ \mathcal{E}(AW_{6}^{+}) $, and their splits contain the splits of $ F_{i} $ as minors, thus the process ends (see Figure \ref{t25}).

\input{25.tpx}

Finally, for $ AW_{2n} (n\geq4) $, it is easy to verfy (by induction) that both of $ AW_{2n} $ with an additional edge of distance 3 on the cycle and and split of $ AW_{2n} $ contain $ V_{8}+e $ minors. And $ AW_{2n} $ with additional edges of distance 2 either is planar or contains a $ V_{8}+e $-minor. Thus let $ H_{1}\cong AW_{2n}^{+}$, and $ C_{2n} $ be a cycle with vertex set $ \left\{ v_{1}, v_{2}, ... , v_{2n}\right\} $ in $ AW_{2n}^{+} $. Then we state the following result, which completes the proof of Lemma 2.1.

\noindent\textbf{Claim 2. }\textit{No addition to $ AW_{2n}^{+} $, $ n\geq4 $ is possible.}

Up to symmetry, we can join two vertices of distance 3 on the cycle $ C_{2n} $, denoted by \textit{type}-$ \uppercase\expandafter{\romannumeral1} $, and two vertices of distance 2 ($ uv_{j} $, $ vv_{i} $ or $ v_{i}v_{i+2} $), such that every two edges construct an $ X $ (reduce the graph with size 8 by contracting edges on $ C_{2n} $, but not in $ X $), denoted by \textit{type}-$ \uppercase\expandafter{\romannumeral2} $, and any two edges do not construct an $ X $, denoted by \textit{type}-$ \uppercase\expandafter{\romannumeral3} $. It is easy to verfy that $ AW_{2n}^{+} $ with additional edges of type-I and type-II give rise to a $ V_{8}+e $-minor. Therefore we assume that only edges of type-III is possible. 

\noindent{\bf Case 1.} Adding an edge $ uv_{j} $ or $ vv_{i} $ to $ AW_{2n}^{+} $, say $ uv_{j} $, generates a cubic vertex $ v_{j-1} $ (or $ v_{j+1} $) in a triangle, thus not $ i $-$ 4 $-$ c $. We first split its neighbors $ u $ (contains the split of $ AW_{6}^{+} $ as a minor) and $ v_{j} $ (contains $ \mathcal{P}_{2} $ as a minor) respectively, thus the process ends. Secondly we have to add an edge $ v_{j-1}v_{j-3} $ (or $ v_{j+1}v_{j+3} $), making the cubic vertex $ v_{j-2} $ (or $ v_{j+2} $) in a triangle, and its split contains $ \mathcal{P}_{1} $ as a minor, so does as adding an edge $ vv_{j-1} $ (or $ vv_{j+1} $),  then the process ends.

\noindent{\bf Case 2.} Adding an edge $ v_{i}v_{i+2} $ makes the cubic vertex $ v_{i+1} $ in a triangle, then no addition is possible and the split of $ v_{i} $ contains $ \mathcal{P}_{1} $ as a minor, thus the process ends.

Therefore no addition to $ AW_{2n}^{+} $ is possible.
\end{proof}

\section{Extension of Terrahawk}

In this section, we shall characterize the internally 4-connedted $ V_{8}+e $-minor-free graphs which are obtained from terrahawk by repeatedly adding edges or splitting vertices.

\begin{lem}\label{lem3}
Let $ G\succeq $terrahawk be an internally 4-connected graph. Then $ G $ is $ V_{8}+e $-minor-free if and only if $ G $ is planar.
\end{lem}

\begin{proof}
To simplify our notation, we denote the terrahawk by $ Ter $, then $ Ter $ is clearly $ V_{8}+e $-minor-free. Thus let $ H_{0}\cong Ter $. Now we analyze graphs obtained from $ Ter $ by splitting vertices or adding edges. 

By symmetry, there are two new non-planar graphs by splitting vertex 1 and 2, respectively, both of which have a $ V_{8}+e $-minor.
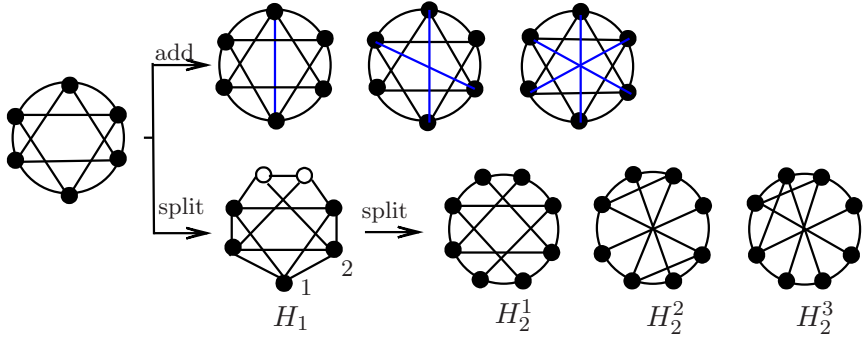
What's more, any non-planar addition of $ Ter $ contains $ V_{8}+e $ as a minor (as shown in Figure \ref{t6}). Similer to the analysis of Lemma 2.1, $ Ter $ with an additional $ X $ also contains a $ V_{8}+e $-minor. Thus both $ H_{1} $ and $ H_{2} $ contain a $ V_{8}+e $-minor, so does any $ H_{i} $ $ (i\geq3) $. Therefore, there is no internally 4-connected, non-planar $ V_{8}+e $-minor-free graph generated from $ Ter $. 

\input{6.tpx} 

\end{proof}

\section{Extension of $ C_{n}^{2} $ and $ K_{3,3} $}
In this section, we shall characterize the internally 4-connedted $ V_{8}+e $-minor-free graphs which are obtained from $ C_{n}^{2} $ and $ K_{3,3} $ by repeatedly adding edges or splitting vertices.

Let $ (X_{1}, X_{2}) $ be a partition of the vertex set of $ K_{3,3} $ such that each vertex of $ X_{1} $ is adjacent to each vertex of $ X_{2} $. Define  $ K_{3,3}^{i,j} $, $ i, j=0,1,2,3 $, to be the graph obtained by adding $ i $ edges between vertices of $ X_{1} $ and $ j $ edges between vertices of $ X_{2} $. As every vertex of $ K_{3,3} $ is symmetric to each other, this process is well-defined. We note a new few things about this notation. First,  $ K_{3,3}^{i,j} $ is ismorphic to  $ K_{3,3}^{j,i} $, and second, $ K_{3,3}^{3,3} $ is $ K_{6} $. Together, we call these 10 graphs $ \mathbb{K}_{3,3} $ (as shown in Figure \ref{t13}).

\input{13.tpx}

\begin{lem}
	Let $ G\succeq C_{n}^{2} $ be an internally $ 4 $-connected graph. Then $ G $ is $ V_{8}+e $-minor-free if and only if $ |G|\leq7 $ or $ G\in$$\left\{\right.$internally $ 4 $-connected minors of $ \Gamma\left.\right\}\cup\mathcal{E}(AW_{6}^{+})\cup\mathcal{E}(V_{8}+f) $.
\end{lem}

\begin{proof} We distinguish between two cases depend on the planarity of $ C_{n}^{2} $.
	
	{\bf Case 1.} $ H_{0} $ is planar, that is $ H_{0}\in C_{2n}^{2} (n\geq3) $.
	
	If $ |H_{0}|\geq8 $, since each $ C_{2n}^{2} $ contains $ C_{8}^{2} $ as a minor, let $ H_{0}\cong C_{8}^{2} $. Now we generate internally 4-connected $ V_{8}+e $-minor-free graphs from $ C_{8}^{2} $.
	
	We claim that every non-planar graph generated from $ C_{8}^{2} $ contains $ V_{8}+e $ as a minor. Let $ H_{1} $ be an $ i $-$ 4 $-$ c $ $ V_{8}+e $-minor-free graph. First suppose that $ H_{1} $ is a split of $ C_{8}^{2} $. By symmetry, we just consider the non-planar split of vertex 1, which gives rise to a $ V_{8}+e $-minor. Next suppose that $ H_{1} $ is an addition of $ C_{8}^{2} $. Then $\left\{14, 16, 27, 25, 36, 38, 47, 58\right\}\subseteq \mathcal{F}(C_{8}^{2})$, which means that only the planar graphs in this process avaliable (as shown in Figure \ref{t7}).
	
	\input{7.tpx}
	
	If $ |H_{0}|<8 $, it is clear that $ H_{0} $ is $ V_{8}+e $-minor-free, then we construct the non-planar $ V_{8}+e $-minor-free graphs based on $ C_{6}^{2} $.
	
	We first obtain three $ i $-$ 4 $-$ c $ $ V_{8}+e $-minor-free graphs by adding edges, which are exactly isomorphic to $ K_{3,3}^{2,2}, K_{3,3}^{2,3} $ and $ K_{6} $, and a split of $ C_{6}^{2} $, denoted by $ H_{1} $ (see Figure \ref{t8}). 
	
	Firstly, by symmetry of $ H_{1} $, we can continue splitting vertices 1 and 2, generating the graphs $ H_{2}^{j} (1\leq j\leq3) $. Obviously, $ H_{2}^{1}\in \mathcal{E}(AW_{6}^{+}) $, thus we construct graphs based on $ H_{2}^{2}$ and $ H_{2}^{3}$. Since both $ H_{2}^{2}$ and $ H_{2}^{3}$ contain $ V_{8}+f $ as a minor, we denote the $ i $-$ 4 $-$ c $ $ V_{8}+e $-minor-free graphs with $ \left|V(G)\right|=8 $ obtained from them by $ \mathcal{E}(V_{8}+f) $ (see Figure \ref{t12}). Then we split vertices of them. Note that the split of $ H_{1} $ with additional edges always contains $ H_{2}^{j} $ as a minor, thus we would not analyze any more, and we will not explicitly explain every time.
	
	\input{8.tpx}
	
	\input{12.tpx}	
	
	{\bf case 1a.} Splits of $ H_{2}^{2} $.
	
	By symmetry of $ H_{2}^{2} $, we split the vertex 8, generating 3 graphs with a $ V_{8}+e $-minor (see Figure \ref{t9}).
	
	\input{9.tpx}
	
	{\bf case 1b.} Splits of $ H_{2}^{3} $. 
	
	By symmetry of $ H_{2}^{3} $, we split the vertex 1 and 2, generating graphs with a $ V_{8}+e $-minor, with an exception $ \Gamma $ (see Figure \ref{t10}).
	
	\input{10.tpx}
	
	Then we characterize graphs from $ \Gamma $ by splitting vertices or adding edges. Clearly, any splits of $ \Gamma $ contains a graph generated from $ H_{2}^{3} $ by splitiing vertex 1 as a minor, thus contains a $ V_{8}+e $-minor. Then we claim that adding any edge to $ H_{3} $ gives rise to a $ V_{8}+e $-minor. Actually by symmetry, we can add edges $ \left\{16, 2^{'}6, 2^{'}7, 2^{''}4, 2^{''}5, 2^{''}7, 35, 36, 46, 47, 57, 58, 68\right\} $. Since $ \left\{14, 16, 25, 27, 36, 38, 47, 58\right\}\subseteq\mathcal{F}(H_{2}^{3})\subseteq\mathcal{F}(\Gamma) $, only edges $ 2^{'}6, 2^{''}4, 35, 46, 57, 68 $ is possible, to which also generates a $ V_{8}+e $-minor the $ \Gamma $ adding (see Figure \ref{t11} for an illustration). Therefore the process terminates at $ \Gamma $.
	
	\input{11.tpx}
	
	Secondly since $ K_{3,3}^{2,2}\preceq K_{3,3}^{2,3}\preceq K_{6} $, let $ H_{1}\cong K_{3,3}^{2,2} $.
	
	Now we consider the graphs obtained from $ K_{3,3}^{2,2} $ by splitting vertices. By symmertry, we can split the vertices 1 and 2.
	
	{\bf Case 1.1.} Split the vertex 1.
	
	Note that the splitting of 1 gives rise to two graphs $ F_{1} $ and $ F_{2} $, both of which are not $ i $-$ 4 $-$ c $. We first make them $ i $-$ 4 $-$ c $ by splitting the neighbors of the cubic vertex. Observe that, by symmetry of $ F_{1} $, splitting vertex 6 generating 3 new not $ i $-$ 4 $-$ c $ graphs with at least 2 cubic vertices in different triangles. However, there is just one time to construct $ H_{1} $ from these 3 graphs, and there is no $ i $-$ 4 $-$ c $ graph obtained from them. In addition, splitting vertex 5 generating either graphs contain $ V_{8}+e $ minors or graphs with 2 cubic vertices in distinct triangles, thus the process ends.
	
	By symmetry of $ F_{2} $, we split vertex 2, generating graphs with $ V_{8}+e $ minors, $ H_{2}^{1} $-minor or $ H_{2}^{3} $-minor (in Figure 19), and a not $ i $-$ 4 $-$ c $ graph $ F_{2}^{1} $. We further split vertex 5 of $ F_{2}^{1} $, since $ F_{2}^{1} $ contains 2 cubic vertices in a triangle, generating either graphs contain $ V_{8}+e $-minor or graphs with 2 cubic vertices in different triangles, similar to the discussion above, the process ends.
	
	Next we make $ F_{1} $ and $ F_{2}$ $ i $-$ 4 $-$ c $ by adding some edges, generating three $ i $-$ 4 $-$ c $ graphs $ \left\{H_{2}^{1}, H_{2}^{2},\right.\\ \left.H_{2}^{3}\right\} =\mathcal{C}$ (see Figure \ref{t14}). 
	
	\input{14.tpx}
	
	Now we construct graph $ H_{3} $ from graphs in $ \mathcal{C} $. Observe that $ |V(H_{2}^{i})|=7 $, and we would not give the additions of them in detail. Therefore, we suppose that $ H_{3} $ is a split of $ H_{2}^{i} $. 
	
	{\bf case 1.1a.} Splits of $ H_{2}^{1} $.
	
	By symmertry, we can split the vertices $ 1^{'}, 2, 3, 4 $ of $ H_{2}^{1} $, respectively. We states that any split of $ H_{2}^{1} $ contains $ V_{8}+e $ as a minor (see Figure \ref{t15} for an illustration).
	
	\input{15.tpx}
	
	{\bf case 1.1b.} Splits of $ H_{2}^{2} $.
	
	By symmetry, we can split the vertices $ 1^{'}, 2, 3, 4 $ of $ H_{2}^{2} $, generating graphs with $ V_{8}+e $ minors, or isomorphic to graphs in $ \mathcal{E}(AW_{6}^{+}) $ (shown in Figure \ref{t16}).
	
	\input{16.tpx}
	
	{\bf case 1.1c.} Splits of $ H_{2}^{3} $.
	
	By symmetry, we can split the vertices $ 1^{'}, 2, 4 $ of $ H_{2}^{3} $, generating graphs with $ V_{8}+e $ minors, or graphs with $ H_{2}^{1} $ and $ H_{2}^{2} $ (in Figure 11) minors (shown in Figure \ref{t17}).
	
	\input{17.tpx}
	
	{\bf Case 1.2.} Split the vertex 2.
	
	Note that the splitting of vertex 2 gives rise to an $ i $-$ 4 $-$ c $ graph $ H_{2}^{4} $, and two not  $ i $-$ 4 $-$ c $ graphs $ F_{3} $ and $ F_{4} $. Splitting $ H_{2}^{4} $ generates two graphs in $ \mathcal{E}(AW_{6}^{+}) $ (up to symmetry). Similar to the Case 1.1, we first consider the split of them. By symmetry of $ F_{3} $, we split vertex 1, generating 2 graphs with $ V_{8}+e $ minors and a not $ i $-$ 4 $-$ c $ graph $ F_{3}^{1} $ with 2 cubic vertices in different triangles, thus the process ends. Then for $ F_{4} $, splitting vertex 1 generates a graph with a $ V_{8}+e $-minor and a graph isomorphic to $ F_{2}^{1} $. Splitting vertex 3, generates 4 graphs $ V_{8}+e $ minors, one in $ \mathcal{E}(AW_{6}^{+}) $, one with a $ H_{2}^{2} $-minor (in Figure 11), and two not $ i $-$ 4 $-$ c $ graphs $ F_{4}^{1} $ and $ F_{4}^{2} $. Since $ F_{4}^{1} $ contains two cubic vertices in a triangle, we split their common neighbor, vertex 1, containing the split of vertex 1 of $ F_{4} $ as a minor. And $ F_{4}^{2} $ contains 2 cubic vertices in different triangles, thus the process ends.
	
	Next, we adding edges to $ F_{3} $ and $ F_{4} $, generating three $ i $-$ 4 $-$ c $ graphs $ H_{2}^{5}, H_{2}^{6} $ and $ H_{2}^{7} $ (see Figure \ref{t18}). 
	
	\input{18.tpx}
	
	Similar to the Case 1.1, we just consider the splits of these graphs. The proof for splits of these graphs is of the same flavor for set $ \mathcal{C} $, see Figure \ref{t19}, \ref{t20}, \ref{t22}. Note that the split of $ H_{2}^{6} $ gives rise to two not $ i $-$ 4 $-$ c $ graphs $ F_{5} $ and $ F_{6} $, similar to the Case 1.1, we stop the process.
	
	\input{19.tpx}
	
	\input{20.tpx}
	
	\input{22.tpx}
	
	Then for graphs in $ \mathcal{E}(AW_{6}^{+}) $ and contain $ H_{2}^{i} $ minors (in Figure 11), have been analyzed in section 2, we stop the process. Therefore, if an internally 4-connected graph $ G $ is $ V_{8}+e $-minor-free, then either $ |G|\leq7 $ or $ G\in\mathcal{E}(AW_{6}^{+})\cup\Gamma $.
	
	{\bf Case 2.} $ H_{0} $ is non-planar, that is $ H_{0}\in C_{2n+1}^{2} (n\geq2) $. 
	
	If $ |H_{0}|\geq8 $, it is suffice to consider graph $ C_{9}^{2} $, which contains a $ V_{8}+e $-minor. Thus suppose that $ |H_{0}|<8 $. 
	
	If $ H_{0}\cong C_{7}^{2} $, then any split of $ C_{7}^{2} $ produces a $ V_{8}+e $-minor, and every addition of $ C_{7}^{2} $ is internally 4-connected $ V_{8}+e $-minor-free, and any split of them contains the split of $ C_{7}^{2} $ as minor, thus contains a $ V_{8}+e $-minor.
	
	If $ H_{0}\cong C_{5}^{2} $, since $ C_{5}^{2} $ is complete, let $ H $ be a graph generating from $ C_{5}^{2} $ by splitting a vertex, which is isomorphic to $ K_{3,3}^{1,1} $. However, the internally 4-connected graphs $ H_{1}^{j} $, generating from $ K_{3,3}^{1,1} $ by adding edges, is isomorphic to $ K_{3,3}^{2,2}, K_{3,3}^{2,3} $ and $ K_{6} $, same as Case 1. Note that any split of $ K_{3,3}^{1,1} $ contains 2 vertices in distinct triangles, thus the process ends.
	
\end{proof}

\begin{lem}
Let $ G\succeq K_{3,3} $. Then $ G $ is $ V_{8}+e $-minor-free if and only if $ G\in \left\{\right.internally$ $4$-connected graphs on $\leq7$ vertices$\left.\right\}\cup \left\{\right.$internally $  4$-connected minors of $\Gamma_{2}\left.\right\} $, where $ \Gamma_{2} $ is shown in Figure \ref{t23}.
\end{lem}

\begin{proof}
Since $ K_{3,3} $ is clearly $ V_{8}+e $-minor-free, let $ H_{0}\cong K_{3,3} $. Then no vertex of $ K_{3,3} $ can be splitted, thus we analyze the additons of $ K_{3,3} $. 

\input{23.tpx}

Let $ F_{1}\cong K_{3,3}^{1,0} $. Since it is not $ i $-$ 4 $-$ c $, we continue adding edges or splitting vertices. By symmetry, we obtain additions $ K_{3,3}^{1,1} $ and $ K_{3,3}^{2,0} $ and a split of $ F_{1} $, namely $ F_{1}^{1} $, both of which are not $ i $-$ 4 $-$ c $. Continuing the process based on $ F_{1}^{1} $ gives rise to an unique $ i $-$ 4 $-$ c $ graph $ H_{1} $, which is isomorphic to $ V_{8} $, and $ \left\{14, 25, 36, 47, 58, 61, 72, 83\right\}\subseteq\mathcal{F}(V_{8}) $. Since $ V_{8} $ contains no vertex of degree 4, we obtain the graph $ V_{8}+f $. We first consider the addition of $ V_{8}+f $, then by symmetry, we obtain a not $ i $-$ 4 $-$ c $ graph $ V_{1} $, $ H_{2}^{3} $ in Figure 11 and others in the set $ \mathcal{E}(V_{8}+f) $. For $ V_{1} $, the addition of it is also in the set $ \mathcal{E}(V_{8}+f) $, so we just able to split the vertex 8, since it contains two cubic vertices with a commom neighbor, generating a graph with a $ V_{8}+e $ minor and 5 not $ i $-$ 4 $-$ c $ graphs , thus the process ends (see Figure \ref{t23}). 

Next we consider the split of $ V_{8}+f $, generating two graphs with $ V_{8}+e $ minors, and an $ i $-$ 4 $-$ c $ $ V_{8}+e $-minor-free graph $ \Gamma_{1} $. Then further splitting vertex and adding edge to $ \Gamma_{1} $. By symmetry, $ \left\{63, 68, 16, 2^{'}6, 2^{''}5, 2^{'}7, 2^{''}7, 2^{''}1, 2^{'}8 \right\}\subseteq\mathcal{F}(\Gamma_{1}) $, then only edge $ 2^{''}4, 2^{''}3 $ or 17 is possible, genereting two graphs with $ V_{8}+e $ minors and a graph isomorphic to $ \Gamma $. Splitting vertex 8 gives rise to a graph with a $ V_{8}+e $-minor, and an $ i$-$4$-$c $ $ V_{8}+e $-minor-graph $ \Gamma_{2} $. Finally, since $ \mathcal{F}(\Gamma_{1})\subseteq\mathcal{F}(\Gamma_{2}) $, we deduce by symmetry that no addition to $ \Gamma_{2} $ is $ V_{8}+e $-minor-free. On the other hand, $ \Gamma_{2} $ is cubic so no split is possible either, therefore the process terminates at $ \Gamma_{2} $.

Note that splitting any vertex of $ K_{3,3}^{1,1} $ remains a cubic vertex in a triangle, so does as splitting vertex 1 of $ K_{3,3}^{2,0} $, and split vertex 2 of $ K_{3,3}^{2,0} $ is exactly an addition of $ H_{1} $, thus not internally 4-connected. Then the process terminates.

\end{proof}

\noindent\textit{\textbf{Proof of Theorem 1.5.}} From Lemmas 2.1, 3.1, 4.1 and 4.2, we obtain a characterization of all internally 4-connected $ V_{8}+e $-minor-free graphs.\qed

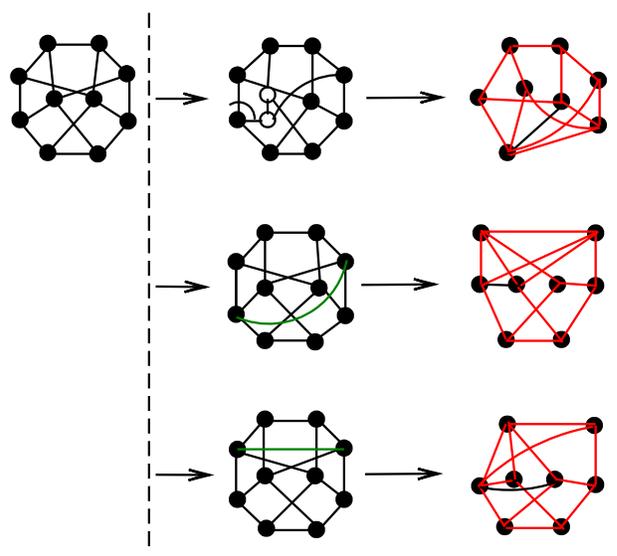
\end{document}